\documentclass[11pt, a4paper, english, reqno]{amsart}
\usepackage{amsmath} % utilities for mathematics
\usepackage{amsthm} % utilities for theorem environment
\usepackage{amssymb} % loads mathematical fonts
\usepackage[dvipsnames]{xcolor}
\usepackage{amscd} % utilities for commutative diagrams
\usepackage{mathtools}
\usepackage[normalem]{ulem}

\usepackage{subcaption}% <-- added

\usepackage{array} % core package

\usepackage[cal=boondoxo,scr=euler]{mathalfa}
\usepackage[linktocpage]{hyperref} % permits navigating on the pdf
\usepackage{cleveref} % smart referencing
\usepackage{caption} %
\usepackage{graphics,graphicx} % permits putting images
\usepackage{tikz,tikz-cd} % tool for diagrams
\usepackage{mlmodern}

\usetikzlibrary{fit,matrix,graphs,graphs.standard,calc,shapes.geometric, arrows.meta}
\usetikzlibrary{decorations.pathreplacing,}%angles,quotes
\usetikzlibrary{automata}

\usepackage{enumerate} % permits enumerating using letters or other symbols
\usepackage{xcolor}

\usepackage[colorinlistoftodos,bordercolor=orange,backgroundcolor=orange!20,linecolor=orange,textsize=scriptsize]{todonotes}

\DeclareMathAlphabet{\mathsf}{OT1}{\sfdefault}{m}{n}

\newcommand{\nocontentsline}[3]{}
\newcommand{\tocless}[2]{\bgroup\let\addcontentsline=\nocontentsline#1{#2}\egroup}

\usepackage[margin=1.23in]{geometry}
\usepackage{verbatim}

\usepackage{scalerel}

\makeatletter
\def\dual#1{\expandafter\dual@aux#1\@nil}
\def\dual@aux#1/#2\@nil{\begin{tabular}{@{}c@{}}#1\\#2\end{tabular}}
\makeatother

\makeatletter
\@namedef{subjclassname@2020}{\textup{2020} Mathematics Subject Classification}
\makeatother

\DeclareMathAlphabet{\amathbb}{U}{bbold}{m}{n}

\hypersetup{
    colorlinks = true,
    linkbordercolor = {white},
    linkcolor = {NavyBlue},
    anchorcolor = {black},
    citecolor = {NavyBlue},
    filecolor = {cyan},
    menucolor = {NavyBlue},
    runcolor = {cyan},
    urlcolor = {black}
}

\usetikzlibrary{automata}

\newtheoremstyle{teoremas}% <name>
{10pt}% <Space above>
{10pt}% <Space below>
{\itshape}% <Body font>
{}% <Indent amount>
{\bfseries}% <Theorem head font>
{}% <Punctuation after theorem head>
{.5em}% <Space after theorem headi>
{}% <Theorem head spec (can be left empty, meaning `normal')>

\theoremstyle{teoremas}
\newtheorem{theorem}{Theorem}[section]

\newtheorem{lemma}[theorem]{Lemma}
\newtheorem{proposition}[theorem]{Proposition}

\newtheoremstyle{definition}% <name>
{10pt}% <Space above>
{10pt}% <Space below>
{}% <Body font>
{}% <Indent amount>
{\bfseries}% <Theorem head font>
{}% <Punctuation after theorem head>
{.5em}% <Space after theorem headi>
{}% <Theorem head spec (can be left empty, meaning `normal')>

\theoremstyle{definition}

\newtheorem{example}[theorem]{Example}
\newtheorem{remark}[theorem]{Remark}

\crefname{theorem}{theorem}{theorems}
\Crefname{theorem}{Theorem}{Theorems}
\crefname{lemma}{lemma}{lemmas}
\Crefname{lemma}{Lemma}{Lemmas}
\crefname{proposition}{proposition}{propositions}
\Crefname{proposition}{Proposition}{Propositions}

\tikzstyle{rectan} = [rectangle, rounded corners, 
minimum width=1.5cm, 
minimum height=0.75cm,
text width=3cm,
text centered, 
draw=black,
font = \normalfont,
]

\tikzstyle{ghost} = [circle, 
minimum width=1pt, 
minimum height=1pt,
text width=1pt,
text centered, 
draw=black,
font = \footnotesize,
]

\DeclareMathOperator{\rk}{rk}

\newcommand{\M}{\mathsf{M}}

\newcommand{\U}{\mathsf{U}}

\newcommand{\ehr}{\operatorname{ehr}}

\AtBeginDocument{%
   \def\MR#1{}
}

\title[The Ehrhart polynomial of a matroid specializes to the $\beta$-invariant]{The Ehrhart polynomial of a matroid\\specializes to the beta invariant}

\author[Chavez]{Anastasia Chavez}
\address{Department of Mathematics and Computer Science, Saint Mary's College of California, Moraga, CA, USA}
\email{amc59@stmarys-ca.edu}

\author[Dorpalen-Barry]{Galen Dorpalen-Barry}
\address{Department of Mathematics, Texas A\&M University, College Station, TX, USA}
\email{dorpalen-barry@tamu.edu }

\author[Ferroni]{Luis Ferroni}
\address{Dipartimento di Matematica, Universit\`a di Pisa, Pisa, Italy}
\email{luis.ferroni@unipi.it}

\author[Liu]{Fu Liu}
\address{Department of Mathematics, University of California, Davis, Davis, CA, USA}
\email{fuliu@ucdavis.edu}

\author[Rinc\'on]{Felipe Rinc\'on}
\address{School of Mathematical Sciences, Queen Mary University of London, London, UK}
\email{f.rincon@qmul.ac.uk}

\author[Vindas-Mel\'endez]{Andr\'es R. Vindas-Mel\'endez}
\address{Department of Mathematics, Harvey Mudd College, Claremont, CA, USA}
\email{avindasmelendez@g.hmc.edu}

\begin{document}

\allowdisplaybreaks

\begin{abstract}
    We show that the linear coefficient of the Ehrhart polynomial of a matroid base polytope evaluated at $t-1$ is equal to, up to normalization, the $\beta$-invariant of the matroid.
    This yields a lattice-point counting formula for the $\beta$-invariant and establishes a new and unexpected positivity property of Ehrhart polynomials of matroid polytopes.
\end{abstract}

\subjclass[2020]{Primary: 05B35, 05A15, 52B05, 52B20, 52B40, 05A10} %13D40, 14C15. 
%Secondary: 16S37}

\keywords{matroids, beta invariant, lattice-point enumeration, Ehrhart polynomial}

\maketitle

\section{Introduction}

Let $\M$ be a matroid with ground set $E=[n]$ and set of bases $\mathscr{B}$. 
One can encode $\M$ as a lattice polytope $\mathscr{P}(\M)\subseteq \mathbb{R}^n$ defined by
    \[ \mathscr{P}(\M) := \text{convex hull} \{ e_B: B\in \mathscr{B}\},\]
where $e_B := \sum_{i\in B} e_i$ for each $B\in \mathscr{B}$, and each $e_i$ denotes a vector of the canonical basis of $\mathbb{R}^n$. 
Through this geometric perspective, matroid functions can be viewed as maps on these polyhedra.
A condition on matroid functions that is particularly useful in the language of base polytopes is that of \emph{valuativity}. 
Roughly speaking, a matroid function is valuative whenever it respects inclusion-exclusion under matroid polytope subdivisions (see \cite{ardila-fink-rincon}).
In recent years, there has been a myriad of articles studying valuative functions on matroids (see \cite{ardila-fink-rincon,derksen-fink,ardila-sanchez,ferroni-schroter,eur-huh-larson,ferroni-fink}).
Often, mathematicians focus on matroid functions that are invariants under matroid isomorphisms---such valuative functions are called \emph{valuative invariants} of matroids. 
Despite being a non-obvious property, most matroid invariants are indeed valuative.
However, the reason why valuativity holds is not often apparent.

In this short article we are concerned with two matroid invariants, both of which are valuative: the \emph{Ehrhart polynomial} of the matroid polytope and the \emph{$\beta$-invariant} of the matroid.
In what follows we recapitulate some facts about these two invariants.

The $\beta$-invariant of a matroid $\M$ was introduced by Crapo \cite{crapo}.
It associates to every matroid $\M$ the number
\[ \beta(\M) := (-1)^{\rk(\M)} \sum_{A\subseteq E} (-1)^{|A|} \rk(A).\]
This is a fundamental statistic associated to the matroid $\M$, and it has been extensively used and studied in the literature.
Among other contexts, it appears naturally when studying hyperplane arrangements \cite{zaslavsky,greene-zaslavsky}, graph theory \cite{crapo,brylawski,oxley-beta}, simplicial complexes \cite{bjorner}, and even tropical geometry \cite{ardila-eur-penaguiao, lopez-rincon-shaw}. 
Even though it is not a priori obvious from its definition, the $\beta$-invariant of a matroid is a nonnegative integer. 
Indeed, $\beta(\M)=1$ if $\M$ is a single coloop, $\beta(\M)=0$ if $\M$ is a single loop or $\M$ is disconnected, and the following deletion-contraction recursion holds:
    \[ \beta(\M) = \beta(\M\setminus i) + \beta(\M/i) \enspace \text{whenever $i\in E$ is not a loop nor a coloop.}\]
If $|E|\geq 2$, then $\beta(\M) = 0$ if and only if the matroid $\M$ is disconnected, and $\beta(\M) = 1$ if and only if the matroid $\M$ is a series-parallel matroid (see \cite{brylawski}).

The Ehrhart polynomial of a matroid polytope, by contrast, has been systematically studied only in the last two decades. 
The Ehrhart polynomial of a matroid $\M$ is defined as the polynomial $\ehr(\M,t)\in \mathbb{Q}[t]$ such that
    \[ \ehr(\M,t) = \#(t\mathscr{P}(\M)\cap \mathbb{Z}^n)\]
for every positive integer $t$.
%That is, we write ``the Ehrhart polynomial of the matroid $\M$” to mean the “Ehrhart polynomial of the matroid polytope of $\M$.”
The fact that this counting function is indeed interpolated by a polynomial in $t$ is a classical result due to Ehrhart \cite{ehrhart}.
We refer to the monograph \cite{beck-robins} for a friendly introduction to this topic.

Much of the interest around Ehrhart polynomials of matroids has revolved around conjectures posed by De Loera, Haws, K\"oppe \cite{deloera-haws-koppe}. 
Recently, the conjecture asserting the positivity of the coefficients of $\ehr(\M,t)$ was disproved in \cite{ferroni}. 
However, many positivity properties have been established (see \cite{castillo-liu,castillo-liu2,jochemko-ravichandran,ferroni1,ferroni-morales-panova}).

The main result of this article reveals a surprising relationship between the $\beta$-invariant of a matroid and the Ehrhart polynomial of its base polytope.

\begin{theorem}\label{thm:main}
    Let $\M$ be a matroid of rank $k$ on $n$ elements, without loops or coloops. 
    The $\beta$-invariant of $\M$ can be obtained from the Ehrhart polynomial of $\mathscr{P}(\M)$ as:
    \[ \frac{\beta(\M)}{(n-1) \binom{n-2}{k-1}} = \frac{\mathrm{d}}{\mathrm{d}t} \thinspace \ehr(\M, t)  \bigg{|}_{t=-1} = [t^1] \ehr(\M, t - 1).\]
\end{theorem}
Note that a matroid on $n\geq 2$ elements that has loops or coloops is disconnected and hence its $\beta$-invariant vanishes.

Theorem \ref{thm:main} allows us to derive several previously unknown facts about Ehrhart polynomials and $\beta$-invariants of matroids.
First, since the $\beta$-invariant is nonnegative for all matroids, the above formula asserts a new positivity property of Ehrhart polynomials of matroids, namely, that the linear coefficient of $\ehr(\M,t-1)$ is nonnegative. 
We point out that Jochemko--Ravichandran \cite{jochemko-ravichandran} and, independently, Castillo--Liu \cite{castillo-liu2} proved that the linear term of $\ehr(\M,t)$, i.e., $\frac{\mathrm{d}}{\mathrm{d}t} \thinspace \ehr(\M,t)\big{|}_{t=0}$, is nonnegative. 
Despite the resemblance, the aforementioned result does not imply nor is implied by ours.

Second, since the valuativity of the Ehrhart polynomial is obvious from the definition, our result gives a simple conceptualization of why the $\beta$-invariant is also a valuative invariant. 
The valuativity of the $\beta$-invariant was originally established in the work of Speyer \cite{tropicallinearspaces} and Ardila--Fink--Rinc\'on \cite{ardila-fink-rincon} via topological reasoning. 
It is worth noting that, although our result \emph{explains} the valuativity of $\beta$, this property of $\beta$ plays a crucial role in our proof.

Third, our result gives a polyhedral way of viewing the $\beta$-invariant.
It is not at all clear that the $\beta$-invariant arises from counting lattice points in the dilation of the matroid base polytope.
To our knowledge, the only result in the literature that is similar in flavor is a formula by Cameron and Fink \cite{cameron-fink} computing the Tutte polynomial---in particular, the $\beta$-invariant---of a matroid by counting lattice points in the Minkowski sum $\mathscr{P}(\M) + x\Delta + y\nabla$, where $\Delta$ and $\nabla$ are the standard simplex and its reflection through the origin, and $x$ and $y$ are integer parameters. 
In our result, the enumeration is performed on dilations of the base polytope, whereas in Cameron--Fink's result, the enumeration is carried out on a subtler object.

\section{The proof}

We assume some familiarity with fundamental terminology of matroid theory and matroid base polytopes, but we refer the reader to~\cite{oxley,Welsh} for a more careful treatment. 
We will also use several standard facts about Ehrhart polynomials, which can be found in \cite{beck-robins}.

We make a brief recapitulation about lattice-path matroids. Fix two nonnegative integers $k \leq n$. We look at lattice paths in the plane $\mathbb{R}^2$ that start at $(0,0)$, end at $(n-k,k)$, and are composed of exactly $n$ unit steps, each being either $+(1,0)$ or $+(0,1)$. 

Each lattice path can be represented as a word of length $n$ in the alphabet $\{N,E\}$, where the $i$-th letter is $N$ if the $i$-th step is $+(0,1)$ (``north'') and $E$ if it is $+(1,0)$ (``east''). See Figure~\ref{fig:lattice-path} for an example of two lattice paths from $(0,0)$ to $(5,4)$. 

\begin{figure}[ht]
\centering
\begin{tikzpicture}[scale=.75, line width=.9pt]

\def\s{1} % unit size

% Coordinates for Lower Path L = EEENENNEN
\coordinate (L0) at (0,0);
\coordinate (L1) at (1,0);
\coordinate (L2) at (2,0);
\coordinate (L3) at (3,0);
\coordinate (L4) at (3,1);
\coordinate (L5) at (4,1);
\coordinate (L6) at (4,2);
\coordinate (L7) at (4,3);
\coordinate (L8) at (5,3);
\coordinate (L9) at (5,4);

% Coordinates for Upper Path U = NEENENNEE
\coordinate (U0) at (0,0);
\coordinate (U1) at (0,1);
\coordinate (U2) at (1,1);
\coordinate (U3) at (2,1);
\coordinate (U4) at (2,2);
\coordinate (U5) at (3,2);
\coordinate (U6) at (3,3);
\coordinate (U7) at (3,4);
\coordinate (U8) at (4,4);
\coordinate (U9) at (5,4);

% Fill the region between the two paths
\fill[blue!20]
(L0) -- (L1) -- (L2) -- (L3) -- (L4) -- (L5) -- (L6) -- (L7) -- (L8) -- (L9)
-- (U9) -- (U8) -- (U7) -- (U6) -- (U5) -- (U4) -- (U3) -- (U2) -- (U1) -- (U0)
-- cycle;

% Draw the full grid from (0,0) to (5,4)
\foreach \x in {0,1,...,4} {
\foreach \y in {0,1,...,3} {
\draw[gray!70] (\x,\y) rectangle ++(1,1);
}
}

% Draw the paths
\draw[blue, thick] (L0) -- (L1) -- (L2) -- (L3) -- (L4) -- (L5) -- (L6) -- (L7) -- (L8) -- (L9);
\draw[red, thick]  (U0) -- (U1) -- (U2) -- (U3) -- (U4) -- (U5) -- (U6) -- (U7) -- (U8) -- (U9);
\end{tikzpicture}
\caption{The paths $L = \text{EEENENNEN}$ and $U = \text{NEENENNEE}$ for $k = 4$ and $n = 9$. The region between them is shaded.}\label{fig:lattice-path} 
\end{figure}

Alternatively, one may encode a path $P$ by the subset $s(P)\subseteq [n]$ consisting of the indices of its north steps. For the two paths in \Cref{fig:lattice-path}, we have $s(L) = \{4,6,7,9\}$ and $s(U) = \{1,4,6,7\}$.

For fixed integers $k$ and $n$, let $L$ and $U$ be two lattice paths of this type. We say that $L$ \emph{lies below} $U$ if, for every $m = 1, \ldots, n$,
\begin{equation}\label{eq:ineq-lattice-paths}
    \bigl| s(L) \cap [m] \bigr| \leq \bigl| s(U) \cap [m] \bigr|\,.
\end{equation}
In such case, we write $L\leq U$.
In \Cref{fig:lattice-path}, for example, $L\leq U$.
It follows from \cite[Theorem~3.3]{bonin-demier} that for every pair of lattice paths $L\leq U$, the collection of sets $\{s(P) : L\leq P \leq U\}$ defines the family of bases of a matroid on $[n]$ of rank $k$. Matroids arising in this way are called \emph{lattice-path matroids}. The skew shape that lies between $L$ and $U$ is often called the \emph{skew-shape representation} of a lattice-path matroid. For example, in Figure~\ref{fig:snake-fence} (on the left) we see the skew-shape representation of the lattice-path matroid of rank $4$ on $n=9$ elements, having $s(L) = \{4,6,7,9\}$ and $s(U) = \{1,4,6,7\}$. It is a well-known fact that a lattice-path matroid is connected if and only if the paths $L$ and $U$ intersect only at $(0,0)$ and $(n-k,k)$ (see \cite[Theorem~3.6]{bonin-demier}).

A \emph{snake} matroid is a connected lattice-path matroid whose skew-shape representation is a {\em border strip}, i.e., it does not contain a $2\times 2$ square (for additional details on snake matroids, see \cite{knauer-martinez-ramirez,ferroni-schroter,benedetti-knauer-valencia,ferroni-morales-panova}).
Every skew shape $\lambda/\mu$ comes equipped with a \emph{cell poset} $P_{\lambda/\mu}$, whose Hasse diagram is obtained from $\lambda/\mu$ by replacing each box with a vertex, connecting adjacent cells, and rotating clockwise by $45$ degrees.
The cell posets of skew shapes corresponding to snake matroids are thus \emph{fence posets} in the sense of \cite[Exercise 3.66]{stanley-ec1}.
In Figure~\ref{fig:snake-fence}, we give an example of a skew shape whose lattice-path matroid is a snake, as well as the corresponding fence poset.
% We refer to \cite[Appendix~A]{ferroni-schroter} for more details. 

Importantly for us, base polytopes of snake matroids are (up to a lattice-point-preserving transformation) the order polytopes of their cell posets; see \cite[Theorem~4.7]{knauer-martinez-ramirez}.
Since the Ehrhart polynomials of order polytopes can be computed from the combinatorics of the poset directly~\cite[Theorem 4.1]{stanley-poset_polytopes}, we can compute the Ehrhart polynomial of a snake matroid using its associated fence poset.

\begin{figure}[h]
\begin{tikzpicture}[scale=.75,line width=.9pt]
  \def\s{1}
  \foreach \i in {0,1,2} {
    \draw (\i*\s, 0) rectangle ++(\s, \s);
  }
  \foreach \i in {0,1} {
    \draw (2*\s + \i*\s, \s) rectangle ++(\s, \s);
  }
  \foreach \j in {0,1,2} {
    \draw (3*\s, \s + \j*\s) rectangle ++(\s, \s);
  }
  \draw (4*\s, 3*\s) rectangle ++(\s, \s);
\end{tikzpicture}
\qquad\qquad\qquad
\begin{tikzpicture}[scale=.75, rotate=-45,line width=.9pt]
  \def\s{1}
  \coordinate (A) at (0.5*\s, 0.5*\s);  % bottom horizontal strip
  \coordinate (B) at (1.5*\s, 0.5*\s);
  \coordinate (C) at (2.5*\s, 0.5*\s);
  \coordinate (D) at (2.5*\s, 1.5*\s);  % upper horizontal strip
  \coordinate (E) at (3.5*\s, 1.5*\s);
  \coordinate (F) at (3.5*\s, 2.5*\s);  % vertical strip
  \coordinate (G) at (3.5*\s, 3.5*\s);
  \coordinate (H) at (4.5*\s, 3.5*\s);  % final rightmost box

  \foreach \i/\j in {A/B, B/C, C/D, D/E, E/F, F/G, G/H} {
    \draw (\i) -- (\j);
  }

  \foreach \p in {A,B,C,D,E,F,G,H} {
    \fill (\p) circle (4pt);
  }
\end{tikzpicture}
\caption{The skew shape $\lambda / \mu = (5,4,4,3) / (3,3,2,0)$ and its cell poset $P_{\lambda/\mu}$. Since $\lambda / \mu$ is a border strip, $P_{\lambda/\mu}$ is a fence poset.}
\label{fig:snake-fence}
\end{figure}

Recently, Ferroni and Schr\"oter \cite{ferroni-schroter} proved that indicator functions of matroid polytopes can be written as linear combinations of indicator functions of direct sums of snake matroids, loops, and coloops.
This result will be a key step in our proofs, so we make it precise below.

For each matroid $\M$ on $E$ of rank $k$, the indicator function of its base polytope is the map $\amathbb{1}_{\mathscr{P}(\M)}:\mathbb{R}^n\to \mathbb{Z}$ defined by
    \[ \amathbb{1}_{\mathscr{P}(\M)}(x) = \begin{cases} 1 & x\in \mathscr{P}(\M) \\0 & x\notin \mathscr{P}(\M) \end{cases}.\]

% The following is a key result that we employ.
% Its proof and its statement are contained in \cite[Theorem~A.7]{ferroni-schroter}.

\begin{theorem}[{{\cite[Theorem~A.7]{ferroni-schroter}}}]\label{thm:snakes-valuative-group}
    Let $\M$ be a matroid. 
    Then there exist matroids $\mathsf{S}_1,\ldots \mathsf{S}_m$ and integers $a_1,\ldots,a_m$ such that:
    \begin{enumerate}[\normalfont(i)]
    \item Each $\mathsf{S}_i$ is a direct sum of snake matroids, loops, and coloops.
    \item The set of loops (resp. coloops) of each $\mathsf{S}_i$ coincides with the set of loops (resp. coloops) of $\M$.
    \item We have $\amathbb{1}_{\mathscr{P}(\M)} = \sum_{i=1}^m a_i\, \amathbb{1}_{\mathscr{P}(\mathsf{S}_i)}$.\label{eq:M-as-sum-of-snakes}
    \end{enumerate}
\end{theorem}

Valuative invariants of matroids factor through indicator functions of polytopes. 
Since the Ehrhart polynomial is a valuative invariant, Theorem~\ref{thm:snakes-valuative-group}(\ref{eq:M-as-sum-of-snakes}) implies
    \begin{equation}\label{eq:ehrhart}
    \ehr(\M,t) = \sum_{i=1}^m a_i \ehr(\mathsf{S}_i,t).
    \end{equation}
This equality will allow us to prove Theorem~\ref{thm:main} by focusing on Ehrhart polynomials of (direct sums of) snake matroids.
%This is the content of the next proposition.

\begin{proposition}\label{pro:linearforsankes}
    Let $\mathsf{S}$ be a snake matroid on $n$ elements and rank $k$. 
    Then,
\[ \frac{\mathrm{d}}{\mathrm{d}t} \thinspace \ehr(\mathsf{S}, t) \bigg{|}_{t=-1} = [t^1] \ehr(\mathsf{S}, t - 1) = \frac{1}{(n-1) \binom{n-2}{k-1}}.\]
\end{proposition}

\begin{proof}
Let $P(\mathsf{S})$ be the cell poset of the skew-shape representation of $\mathsf{S}$, and let $\Omega(P(\mathsf{S}),t)$ be the number of order-preserving maps $P \to \{1,\dots,t\}$.
It is well-known that $\Omega(P(\mathsf{S}),t)$ is a polynomial in $t$ (this is true for any poset); see \cite[Section 3.12]{stanley-ec1} for details.
By \cite[Proposition~4.1]{ferroni-morales-panova}, the linear term of $\Omega(P(\mathsf{S}),t)$ is
\[ \frac{(n-k-1)!(k-1)!}{(n-1)!} = \frac{1}{(n-1)\binom{n-2}{k-1}}.\]      
Combining \cite[Theorem 4.1]{stanley-poset_polytopes} and \cite[Theorem~4.7]{knauer-martinez-ramirez} gives
% The Ehrhart polynomial of a snake matroid is known to be equal to the Ehrhart polynomial of the order polytope of the cell poset of the skew-shape representation as a lattice-path matroid; see \cite[Theorem~4.7]{knauer-martinez-ramirez}.
% Specifically, if we call $P(\mathsf{S}_i)$ the cell poset of $\mathsf{S}_i$, we have
        \[ \Omega(P(\mathsf{S}),t) = \ehr(\mathsf{S},t-1)\,,\]
so in particular their linear terms agree.
\end{proof}

\begin{remark}
    At the moment we do not know of a more conceptual way to explain why the quantity computed in the preceding proposition does not depend on the snake matroid, but only relies on its size and rank. 
    The computation by Ferroni, Morales, and Panova in \cite{ferroni-morales-panova} relies on a determinantal formula due to Kreweras for enumerating skew plane partitions. 
\end{remark}

\begin{lemma}\label{lem:disconnected}
    Let $\M$ be a loopless and coloopless disconnected matroid. 
    Then,
    \[ \frac{\mathrm{d}}{\mathrm{d}t} \thinspace \ehr(\M, t)  \bigg{|}_{t=-1} = [t^1] \ehr(\M, t - 1) = 0.\]
\end{lemma}

\begin{proof}
Suppose $\M = \M_1\oplus\cdots \oplus\M_s$ is the decomposition of $\M$ as a direct sum of connected matroids, where $s\geq 2$. 
The base polytope of $\M$ decomposes as the Cartesian product $\mathscr{P}(\M) = \mathscr{P}(\M_1) \times \cdots \times \mathscr{P}(\M_s)$. 
    Taking Ehrhart polynomials, we get
        \[ \ehr(\M,t) = \prod_{i=1}^s \ehr(\M_i,t).\]
    Since each $\mathscr{P}(\M_i)$ is a $0/1$-polytope, it does not contain lattice points in its relative interior.
    Using Ehrhart--Macdonald reciprocity (see \cite[Chapter~4]{beck-robins}), this means that $\ehr(\M_i,-1)=0$. The only exception would be when $\mathscr{P}(\M_i)$ is a point, but that case is ruled out because it corresponds to $\M_i$ being a direct sum of loops and coloops.
    Since $s\geq 2$ and each factor $\ehr(\M_i,t)$ vanishes at $t=-1$, we have that $-1$ is a root of $\ehr(\M,t)$ of multiplicity at least $2$, which yields the desired equality.
\end{proof}

We are now ready to prove our main theorem.

\begin{proof}[Proof of Theorem~\ref{thm:main}]
    If $\M$ is loopless and coloopless, Theorem~\ref{thm:snakes-valuative-group} says that we can find matroids $\mathsf{S}_1,\ldots,\mathsf{S}_m$, all of which are direct sums of snake matroids, and integers $a_1,\ldots,a_m$ satisfying the equality in Theorem~\ref{thm:snakes-valuative-group}(\ref{eq:M-as-sum-of-snakes}).
    In particular, Equation~\eqref{eq:ehrhart} holds and we have
    \[[t^1]\ehr(\M,t-1) = \sum_{i=1}^m a_i \thinspace [t^1]\ehr(\mathsf{S}_i,t-1).\]
    By Lemma \ref{lem:disconnected}, if $\mathsf{S}_i$ is disconnected then $[t^1]\ehr(\mathsf{S}_i,t-1)=0$, so we can restrict the sum to (connected) snake matroids.

    % Now consider Theorem~\ref{thm:snakes-valuative-group}(\ref{eq:M-as-sum-of-snakes}).
    The $\beta$-invariant is a valuative function, so Theorem~\ref{thm:snakes-valuative-group}(\ref{eq:M-as-sum-of-snakes}) also implies
    \[ \beta(\M) = \sum_{i=1}^m a_i\,\beta(\mathsf{S}_i).\]
    As mentioned before, the $\beta$-invariant vanishes on disconnected matroids, so this sum can also be restricted to (connected) snake matroids. 
    Moreover, since every snake matroid is a series-parallel matroid, we have $\beta(\mathsf{S}_i) = 1$ for any connected $\mathsf{S}_i$ appearing above, see \cite{brylawski,oxley-beta}.
    This gives
    \[\beta(\M) = \sum_{\text{$\mathsf{S}_i$ connected}} a_i.\]
    Using Proposition \ref{pro:linearforsankes} and putting the pieces together, we get
    \begin{align*}
        [t^1]\ehr(\M,t-1) &= \sum_{\text{$\mathsf{S}_i$ connected}} a_i \thinspace [t^1]\ehr(\mathsf{S}_i,t-1)\\
        &= \sum_{\text{$\mathsf{S}_i$ connected}} a_i \thinspace \frac{1}{(n-1)\binom{n-2}{k-1}}\\
        &= \frac{1}{(n-1)\binom{n-2}{k-1}} \sum_{\text{$\mathsf{S}_i$ connected}} a_i\\
        &= \frac{\beta(\M)}{(n-1)\binom{n-2}{k-1}},
    \end{align*}
    as desired.
\end{proof}

\begin{remark}
    Theorem~\ref{thm:main} is false if we drop the condition that the matroid is both loopless and coloopless. 
    For example, let $\M=\mathsf{U}_{2,4}\oplus \U_{1,1}$ be the direct sum of two uniform matroids.
    The $\beta$-invariant of $\M$ is $0$ since $\M$ is disconnected, but the base polytope of $\M$ is integrally equivalent to the base polytope of $\mathsf U_{2,4}$ and thus the linear term of $\ehr(\M,t-1)$ coincides with the linear term of $\ehr(\U_{2,4},t-1)$, which is equal to~$\frac{1}{3}$.
\end{remark}

\begin{example}
	Let us consider the uniform matroid $\U_{3,8}$. The $\beta$-invariant of the uniform matroid $\U_{k,n}$ equals $\binom{n-2}{k-1}$ (see, for example, \cite[Proposition~7]{crapo}), so we have $\beta(\U_{3,8}) = \binom{6}{2} = 15$. Theorem~\ref{thm:main} says that
		\[[t^1] \ehr(\U_{3,8},t-1)= \frac{15}{7 \cdot \binom{6}{2}} = \frac{1}{7}.\]
	This can be verified directly, for instance using closed formulas for the Ehrhart polynomial of a uniform matroid (see for example \cite{katzman,ferroni1}). In particular, \cite[Proposition~2.1]{ferroni1} gives
		\[ \ehr(\U_{3,8},t-1) = \frac{397}{1680} t^{7} + \frac{49}{144} t^{6} + \frac{1}{48} t^{5} + \frac{19}{144} t^{4} + \frac{1}{10} t^{3} + \frac{1}{36} t^{2} + \frac{1}{7} t,\]
	where the linear terms matches the expected value.	
\end{example}

\begin{remark}
    The $\beta$-invariant appears prominently as the linear coefficient of the $g$-polynomial introduced by Speyer in \cite{speyer} (see also \cite{ferroni-speyer,fink-shaw-speyer,berget-fink}).
    This polynomial has been used as a tool to approach the $f$-vector conjecture, posed by Speyer in \cite{tropicallinearspaces}, concerning the maximal number of faces in a matroid polytope subdivision.
    The conjecture was settled very recently by Berget and Fink \cite{berget-fink}, who proved that the top coefficient of the $g$-polynomial---called the $\omega$-invariant of the matroid---is a nonnegative number.
    
    It follows from \cite[Proposition~12.1]{fink-shaw-speyer} that whenever the matroid $\M$ has rank $k$ and size $n=2k$, we have $\omega(\M) = (-1)^{c(\M)}\cdot\ehr(\M,-2)$, where $c(\M)$ denotes the number of connected components of $\M$. 
    Therefore, in this special case where $n=2k$, $\omega(\M)$ is also a specialization of the Ehrhart polynomial of $\M$.
    However, if $n\neq 2k$, the last assertion is no longer true. 
    We have verified by an exhaustive computation that the $\omega$-invariant of matroids on $7$ elements of rank $3$ is not a linear specialization of the Ehrhart polynomial: we found a collection of connected matroids of this rank and size whose Ehrhart polynomials satisfy a linear relation that is not satisfied by the corresponding $\omega$-invariants.
\end{remark}

\section*{Acknowledgments}
This article resulted from a SQuaRE at the American Institute for Mathematics (AIM). 
The authors thank AIM for providing a very supportive and mathematically rich environment, and are grateful to two anonymous reviewers for their helpful comments and suggestions that improved this article.
Anastasia Chavez is partially supported by NSF grant DMS 2332342. Luis Ferroni was a member at the Institute for Advanced Study, funded by the Minerva Research Foundation. Fu Liu is partially supported by an NSF grant \#2153897-0. 
\bibliographystyle{amsalpha}
\bibliography{bibliography}

\end{document}